\documentclass[12pt]{amsart}
\usepackage[mathscr]{eucal}
\usepackage{amssymb,amscd}

\newtheorem{theorem}{Theorem}[section]
\newtheorem{proposition}[theorem]{Proposition}
\newtheorem{lemma}[theorem]{Lemma}

\theoremstyle{definition}
\newtheorem{definition}[theorem]{Definition}

\theoremstyle{remark}
\newtheorem{remark}[theorem]{Remark}
\newtheorem*{remark*}{Remark}

\numberwithin{equation}{section}

\newcommand{\bR}{{\mathbb R}}
\newcommand{\bC}{{\mathbb C}}

\newcommand{\ttT}{\mathtt{T}}

\newcommand{\teven}{\mathrm{even}}
\newcommand{\todd}{\mathrm{odd}}

\newcommand{\tanalytic}{\mathrm{an}}

\newcommand{\cC}{\mathcal{C}}

\newcommand{\cE}{\mathcal{E}}
\newcommand{\cF}{\mathcal{F}}
\newcommand{\cG}{\mathcal{G}}
\newcommand{\cL}{\mathcal{L}}

\newcommand{\cI}{\mathcal{I}}
\newcommand{\cS}{\mathcal{S}}
\def\cT{\mathcal{T}}
\newcommand{\cU}{\mathcal{U}}
\newcommand{\cV}{\mathcal{V}}
\newcommand{\cO}{\mathcal{O}}

\newcommand{\fg}{\mathfrak{g}}
\newcommand{\fk}{\mathfrak{k}}
\newcommand{\fm}{\mathfrak{m}}

\newcommand{\fp}{\mathfrak{p}}
\newcommand{\fs}{\mathfrak{s}}
\newcommand{\fv}{\mathfrak{v}}
\newcommand{\ft}{\mathfrak{t}}
\newcommand{\fh}{\mathfrak{h}}
\newcommand{\fq}{\mathfrak{q}}
\newcommand{\fb}{\mathfrak{b}}
\newcommand{\ff}{\mathfrak{f}}
\def\fa{\mathfrak{a}}

\def\op{\oplus}
\def\a{\alpha}
\def\s{\sigma}
\def\bi{\mathbf{i}}
\def\tand{\quad\hbox{and}\quad}
\def\tAd{\mathrm{Ad}}
\def\tim{\mathrm{im}}

\newcommand{\hsp}[1]{{\hbox{\hspace{#1}}}}
\newenvironment{myeqn}[1]
{ \begin{center}
  $({#1})$ \hfill $\displaystyle
}
{ $ \hfill \hsp{0pt} \end{center} }

\newenvironment{blist}{ 
  \begin{list}{{\small{$\bullet$}}}
   {\usecounter{cnt} \setlength{\itemsep}{2pt}
    \setlength{\leftmargin}{15pt} \setlength{\labelwidth}{20pt} }
   }
   {\end{list}}
\newcounter{acnt}
\newenvironment{a_list}{ 
  \begin{list}{\emph{(\alph{acnt})}}
   {\usecounter{acnt} \setlength{\itemsep}{3pt}
    \setlength{\leftmargin}{25pt} \setlength{\labelwidth}{20pt} }
   }
   {\end{list}}

\begin{document}
\title[Non-classical domains]
{Quotients of non-classical flag domains are not algebraic}

\author[Griffiths]{Phillip Griffiths}
\address{School of Mathematics, Institute for Advanced Study, Einstein Drive, Princeton, NY  08540}
\email{pg@ias.edu}

\author[Robles]{Colleen Robles}
\address{Mathematics Department, Mail-stop 3368, Texas A\&M University, College Station, TX  77843-3368}
\email{robles@math.tamu.edu}

\author[Toledo]{Domingo ~Toledo}
\address{ Mathematics Department,
University of Utah, Salt Lake City, Utah 84112  USA  }
\email{toledo@math.utah.edu }

\thanks{Robles is partially supported by NSF DMS-1006353, Toledo is partially supported by Simons Foundation grant 208853.}

\date{\today}

\begin{abstract}
A flag domain $D = G/V$ for $G$ a simple real non-compact group $G$ with compact Cartan subgroup is non-classical if it does not fiber holomorphically or anti-holomorphically over a Hermitian symmetric space.  We prove that any two points in a non-classical domain $D$ can be joined by a finite chain of compact subvarieties of $D$.  Then we prove that for $\Gamma$ an infinite, finitely generated discrete subgroup of $G$, the analytic space $\Gamma\backslash D$ does not have an algebraic structure.
\end{abstract}

\keywords{Hodge theory, flag domain}
\subjclass[2010]
{
  53C30, 
  58A14. 
}
\maketitle

\section{Introduction}
\label{sec:introduction}

The purpose of this paper is to prove two properties of non-classical flag domains.  The first is a geometric property of the collection of their compact subvarieties:  any two points can be joined by a chain of compact subvarieties.   Then, as an application, we prove that for any infinite,  finitely generated, discrete group $\Gamma$ of automorphisms of  a non-classical domain $D$, the quotient space $\Gamma\backslash D$ is not an algebraic variety.

By \emph{ flag domain} we mean  a homogeneous complex manifold  $D = G/V$ for a non-compact, connected, linear,  real semisimple Lie group $G$ with a compact Cartan subgroup $T$, where the  isotropy group $V$ is the centralizer of a subtorus of $T$. Note that our terminology includes a wider class than the manifolds $G/T$. 

 It will be convenient to assume that $G$ is \emph{simple}.  The formulation and proofs of the corresponding results in the general semi-simple case  easily reduce to the simple case.  It will also be convenient to assume that $G$ is an \emph{adjoint} group, so it is also simple as an abstract group, and therefore acts effectively on its various homogeneous spaces.

 There is a unique maximal compact subgroup $K$ of $G$ containing $V$ and we have inclusions $T\subset V\subset K\subset G$.   We have a fibration $p: G/V\to G/K$  of $D$ to the symmetric space $G/K$ for $G$ with fiber $K/V$.  We say that the domain $D = G/V$ is \emph{classical} if $G/K$ is Hermitian symmetric \emph{and}  $p$ is holomorphic or anti-holomorphic.  Otherwise $D$ is said to be \emph{non-classical}.

The fibers of $p$ are complex subvarieties of $D$, biholomorphic to the homogeneous projective variety $K/V$.  The dichotomy between classical and non-classical can be formulated in many equivalent ways.   The characterization most relevant for us is: $D$ is non-classical if and only if it has  compact subvarieties that are not contained in the fibers of $p$.  In this case there is a family $Z_u\subset D$ of complex subvarieties, which are deformations of the fibers of $p$.  This family is parametrized by a Stein manifold $\cU$,  which contains the symmetric space $G/K$ as a totally real submanifold.  See, for example, \cite{FHW, GCBMS, WW} for more details on $\cU$. 

We prove the following theorem:

\begin{theorem}
\label{thm:Zconnectedness}
Let $D$ be a non-classical domain, with $G$ simple, and let $x,y\in D$.   Then there exists a finite sequence $u_1,\dots,u_k\in \cU$ so that $x\in Z_{u_1}, y\in Z_{u_k}$ and $Z_{u_i}\cap Z_{u_{i+1}}\ne\emptyset$ for $i =1,\dots, k-1$.
\end{theorem}

Observe that this theorem immediately implies  the known fact  (see Theorem 5.7 of \cite{Wolf} or Theorem 4.4.3 of \cite{FHW})
that $D$ has no non-constant holomorphic functions.   Therefore, if $D = G/V$ does not fiber holomorphically or antiholomorphically over $G/K$, then it admits no non-constant holomorphic map to \emph{any} Hermitian symmetric domain.

 As an application of this theorem, we prove:

\begin{theorem}
\label{thm:nonalgebraic}
Let $D$ be a non-classical domain, with $G$ simple and adjoint, and let $\Gamma\subset G$ be an infinite, finitely generated  discrete subgroup.  Then the  normal complex analytic space  $X = \Gamma\backslash D$ does not have an algebraic structure. 
\end{theorem}
 We remark that the most important case of this theorem is when $\Gamma$ has no torsion, in which case $X$ is a complex manifold with $\pi_1(X) = \Gamma$.   If $\Gamma$ has torsion but is finitely generated, then it always has a torsion-free subgroup $\Gamma'$ of finite index, and $X' = \Gamma'\backslash D$ is a complex manifold with a finite map $X'\to X$.   Another important case is when $\Gamma$ is a lattice in $G$.
  
 We say that a complex manifold (or an analytic space) $X$ has an algebraic structure if $X$ is biholomorphic to the analytic space $W^{an}$  associated to an abstract algebraic variety $W$ over $\bC$ (or a scheme $W$ of finite type over $\bC$).  For example, a quasiprojective variety has an algebraic structure in this sense.

The motivation for Theorem \ref{thm:nonalgebraic} comes from the study of variations of Hodge structure and period mappings for families of algebraic varieties $\cV\to S$ parametrized by an algebraic variety $S$.   In this situation there is a  flag domain $D$, the period domain, determined by the Hodge structure of the fibers of $\cV\to S$, and a holomorphic map $S\to X = \Gamma\backslash D$, the period map, where $\Gamma\subset G$ is the monodromy group  (which can be assumed to be torsion-free by passing to a finite cover of $S$).  If the Hodge structure has weight one, then $D$ is a Hermitian symmetric domain, but if it has higher weight, then $D$ is typically a non-classical domain.  

One of the ways it was realized that Hodge structures of higher weight are fundamentally different from those of weight one is that there are no automorphic forms on their period domains, meaning that  $H^0(X,\omega^{\otimes k}_X) = 0$ for all $k>0$.  This led to the general suspicion, never proved, that $X$ is not an algebraic variety.  Theorem \ref{thm:nonalgebraic} confirms that this is indeed the case.

It is worth noting that $X$ contains algebraic subvarieties.  Besides the $Z_u$, it contains, e.g, images of period mappings $S\to X$, and quotients of equivariantly embedded Hermitian symmetric domains $D'\subset D$.  The geometry of such non-algebraic varieties containing many positive-dimensional algebraic subvarieties is a subject of current interest.

Note that Theorem \ref{thm:nonalgebraic} is known in the case that $\Gamma$ is  co-compact  \cite{CT}.     The point of this paper is to treat the non-co-compact case, which seems to be difficult to attack by the methods of \cite{CT}.  It is also the significant case for Hodge theory, since the images of monodromy representations  are typically not co-compact.

Since the subvarieties $Z_u$ are rational,  rational connectedness of $\Gamma\backslash D$ can be derived from Theorem  \ref{thm:Zconnectedness}.   Compact  rationally connected complex manifolds, satisfying suitable conditions (say class $\cC$) to guarantee compactness of cycle spaces, have finite fundamental group, see \S 2  of \cite{Campana}.   It is fortunate for us that the finiteness of the fundamental group remains true even for non-compact rationally connected algebraic varieties, thus proving Theorem \ref{thm:nonalgebraic}.   It was pointed out to us by Koll\'ar that this is an easy consequence of the properties of the Shafarevich map \cite{Kollar}.   He also pointed out that it holds for rationally connected  Zariski open sets in complex spaces of class $\cC$ (bimeromorphic to a  K\"ahler manifold), therefore the conclusion of  Theorem \ref{thm:nonalgebraic}  can be strengthened to say that  $\Gamma\backslash D$ is not a Zariski open set in a complex space of class $\cC$.

We thank J\'anos Koll\'ar for pointing us to Theorem 3.6  of \cite{Kollar} and for several very helpful conversations, comments and concrete suggestions.    We also thank Igor Zelenko  for pointing us to \cite{Jurdjevic}.

\section{Flag Domains}
\label{sec:flagdomains}

\subsection{Lie theory preliminaries} \label{S:LT}

In this section we set notation and review some Lie theory associated to the flag domain $G/V$.  Let $\ft \subset \fv \subset \fk \subset \fg$ denote the Lie algebras of $T \subset V \subset K \subset G$.  Let $\fg = \fk \oplus \fq$ be the Cartan decomposition.  In particular, the Killing form $B$ is negative definite on $\fk$ and positive definite on $\fq$, and 
$$
  [\fk,\fq] \,\subset\, \fq \quad\hbox{and}\quad 
  [\fq,\fq]\,\subset\,\fk \,.
$$
Given a subspace $\fa \subset \fg$, let $\fa_\bC$ denote the complexification.  Then $\fh = \ft_\bC$ is a Cartan subalgebra of $\fg_\bC$.  Let $\Delta = \Delta(\fg_\bC,\fh)$ denote the roots of $\fg_\bC$.  Given a root $\a\in\Delta$, let $\fg^\a \subset \fg_\bC$ denote the associated root space.  Given a subspace $\fs \subset \fg_\bC$, define 
$$
  \Delta(\fs) \ = \ \{ \a\in\Delta \ | \ \fg^\a \subset \fs \} \,.
$$
If $[\fh,\fs] \subset \fs$, then 
$$
  \fs \ = \ \left( \fs \cap \fh \right) \ \op \ 
  \bigoplus_{\a\in\Delta(\fs)} \fg^\a \,.
$$
In particular, 
\begin{eqnarray}
  \nonumber
  \fv_\bC & = & \fh \ \op \ \bigoplus_{\a\in\Delta(\fv_\bC)} \fg^\a \,,\\
  \label{E:vkq}
  \fk_\bC & = & \fh \ \op \ \bigoplus_{\a\in\Delta(\fk_\bC)} \fg^\a \,,\\
  \nonumber
  \fq_\bC & = & \bigoplus_{\a\in\Delta(\fq_\bC)} \fg^\a \,.
\end{eqnarray}
Fix a Borel subalgebra $\fh \subset \fb \subset \fg_\bC$.  As above,
$$
  \fb \ = \ \fh \ \op \ \bigoplus_{\a\in\Delta(\fb)} \fg^\a \,.
$$
Moreover, 
$$
  \Delta^+ \ = \ \Delta(\fb)
$$
determines a set of positive roots.  Let $\{\s_1,\ldots,\s_r\}\subset\Delta^+$ denote the corresponding simple roots.  Let $\{\ttT^1,\ldots,\ttT^r\}$ denote the basis of $\fh$ dual to the simple roots, 
\begin{equation} \label{E:Ti}
  \s_i(\ttT^j) \ = \ \delta^j_i \,.
\end{equation}
Let $\Delta(\fv_\bC)^\perp = \Delta \backslash \Delta(\fv)$.  Define 
\begin{equation} \label{E:T}
  \ttT \ = \ \sum_{\s_i \in \Delta(\fq_\bC)} \ttT^i 
  \ + \ \sum_{\s_i \in \Delta(\fk_\bC)\backslash\Delta(\fv_\bC)} 2\,\ttT^i \,.
\end{equation}
Since $\ttT \in \fh$ is semisimple, the Lie algebra $\fg_\bC$ decomposes as a direct sum of $\ttT$--eigenspaces
\begin{equation} \label{E:espace}
  \fg_\bC \ = \ \fg_k \,\op\, \fg_{k-1} \,\op\cdots\op\,\fg_{1-k} \,\op\,\fg_{-k} \,,
\end{equation}
where
$$
  \fg_\ell \ = \ \{ \xi \in \fg_\bC \ | \ [\ttT,\xi] = \ell\xi \} \,.
$$
Note that, as the roots are integral linear combinations of the simple roots, \eqref{E:Ti} and \eqref{E:T} imply that the eigenvalues $\ell$ are integers.   

\begin{remark*} 
Since the roots are pure imaginary on the compact Cartan subalgebra $\ft$, it follows that $\ttT \in \bi\ft$, where $\bi = \sqrt{-1}$.  Therefore,
\begin{equation} \label{E:conj}
  \overline{\fg_k} \ = \ \fg_{-k} \,.
\end{equation}
It follows from \eqref{E:espace} and \eqref{E:conj} that $\fg^{k,-k} = \fg_k$ defines a real, weight zero Hodge structure on $\fg$.  Moreover, the fact that $\fg = \fk \op \fq$ is a Cartan decomposition implies that this Hodge structure is polarized by the Killing form $B$.
\end{remark*}

 By the Jacobi identity,
\begin{equation} \label{E:gr}
  [\fg_\ell,\fg_m] \ \subset \ \fg_{\ell+m} \,.
\end{equation}
We call $\ttT$ a \emph{grading element}, and the eigenspace decomposition \eqref{E:espace} the \emph{$\ttT$--graded decomposition of $\fg_\bC$}.  Observe that 
\begin{eqnarray*}
  \fh \,\subset\, \fv_\bC & = & \fg_0 \,,\\
  \fk_\bC & = & \fg_\teven \ = \ \oplus \,\fg_{2\ell} \,,\\
  \fq_\bC & = & \fg_\todd \ = \ \oplus \,\fg_{2\ell+1} \,.
\end{eqnarray*}
It will be convenient to write 
$$
  \fg_+ \ = \ \bigoplus_{\ell>0} \fg_\ell \quad \hbox{and} \quad
  \fg_- \ = \ \bigoplus_{\ell>0} \fg_{-\ell}\,.
$$
Likewise, $\fk_\bC = \fk_+ \op \fk_0 \op \fk_-$ and $\fq_\bC = \fq_+ \op \fq_-$, with $\fk_0 = \fg_0$, 
$$
  \fk_\pm \ = \ \fk_\bC \,\cap\, \fg_\pm \quad\hbox{and}\quad
  \fq_\pm \ = \fq_\bC \,\cap\, \fg_\pm \,.
$$

Let 
$$
  \fm_\ell \ = \ \fg \,\cap\, \left( \fg_\ell \op \fg_{-\ell} \right) \,.
$$
Then 
$$
  \fg \ = \ \fv \ \op \ \fm_1 \,\op\cdots\op\,\fm_k \,.
$$
The real tangent bundle $TD$ is the homogeneous vector bundle
$$
  TD \ = \ G \times_V (\fg/\fv) \ = \ G \times_V \fm_+ \,,
$$
where $\fm_+ = \fm_1 \op\cdots\op\fm_k$.  A homogeneous complex structure on $D$ is given by specifying $T_\bC D = T_{1,0}D \op T_{0,1} D$ with 
\begin{equation} \label{E:cpxstr}
  T_{1,0} D \ = \ G \times_V \fg_- \quad\hbox{and}\quad
  T_{0,1} D \ = \ G \times_V \fg_+ \,.
\end{equation}

\subsection{The compact dual}

Equation \eqref{E:gr} implies that 
\begin{equation} \label{E:p}
  \fp \ = \ \fg_{\ge0} \ = \ \fv_\bC \,+\, \fb 
\end{equation}
is a Lie subalgebra of $\fg_\bC$.  Moreover, $\fb \subset \fp$ implies that $\fp$ is a parabolic Lie algebra.  Let $G_\bC$ be the complexification of $G$, and let $P \subset G_\bC$ be the parabolic subgroup with Lie algebra $\fp$.  Then the flag domain $D = G/V$ may be identified with the open $G$--orbit of $P/P$ in the \emph{compact dual} $\check D = G_\bC/P$.  

A $G_\bC$--homogeneous complex structure on $\check D$ is given by
\begin{equation} \label{E:checkCstr}
  T_{1,0} \check D \ = \ G_\bC \times_P (\fg/\fp) \quad\hbox{and}\quad
  T_{0,1} \check D \ = \ G_\bC \times_P (\overline{\fg/\fp}) \,.
\end{equation}
From \eqref{E:gr} and \eqref{E:p}, it is immediate that $\fg/\fp \simeq \fg_+$ is a $\fp$--module identification; thus, $T_{0,1} \check D = G_\bC \times_P \fg_+$.  Likewise, the vector space identification $\fg/\fp \simeq \fg_-$, and the fact that $\fg/\fp$ is a $\fp$--module,  allows us to view $\fg_-$ as a $\fp$--module, and write $T_{1,0} = G_\bC \times_P \fg_-$.  In particular, the restriction of the $G_\bC$--homogeneous complex structure \eqref{E:checkCstr} on $\check D$ to $D$ agrees with the $G_\bR$--homogeneous complex structure \eqref{E:cpxstr} on $D$.

Let
$$
  \cT D \ = \ T_{1,0} D \quad\hbox{and}\quad
  \cT \check D \ = \ T_{1,0} \check D 
$$
denote the holomorphic tangent bundles.

\subsection{Compact subvarieties}

Consider the natural fibration $p:G/V\to G/K$.   The fibre
$$
  Z \ = \ p^{-1}(K/K) \ = \ K/V \,.
$$
is a compact, complex submanifold of $D$.  If $K_\bC$ is the complexification of $K$, then 
$$
  Z \ = \ K_\bC / \left( K_\bC \,\cap\, P \right) \,.
$$

\begin{lemma}
Suppose $G$ is simple, and let $U \subset G_\bC$ be a neighborhood of the identity with the property that $g Z \subset D$ for all $g \in G_\bC$.  Then $D$ is classical if and only if for every $g\in U$, $gZ$ is a fiber of $p$.    
\end{lemma}

\begin{proof}
This is contained in Proposition 2.3.5 of \cite{WW}, where the largest  subgroup $L\subset G$ preserving $Z$  is determined.
\end{proof}

Suppose $D$ is non-classical, and let $U$ be the maximal  connected neighborhood of the identity in $G_\bC$ with $gZ\subset D$ for all $g\in U$.  Observe that $U$ is invariant under the right action of $K_\bC$.   

\begin{definition}
Assume $D$ be non-classical.  Define
$$
  \cU \ = \ U/K_\bC \subset G_\bC/K_\bC \,;
$$
this is the parameter space for the deformations $gZ \subset D$ of $Z$.  Given $u = g K_\bC \in \cU$, let
$$
  Z_u \ = \ gZ \ \subset \ D
$$
denote the corresponding variety.  The \emph{incidence variety} is
$$
   \cI \ = \ \left\{ (x,u) \in D \times \cU \ | \ x \in Z_u \right\} \,.
$$
\end{definition}
The natural projections $\cI\to D$ and $\cI\to \cU$ yield a diagram:
\begin{equation}
\label{eq:incidencediagram}
 \begin{array}{ccccc} &  & \cI &  &  \\ &  \swarrow &  & \searrow  &  \\D &  &  &  & \cU\end{array}
\end{equation} 

The proof of Theorem \ref{thm:Zconnectedness} utilizes three sub-bundles of the the holomorphic tangent bundle $\cT\cI$. First, observe that the holomorphic tangent bundle $\cT\cI\subset \cT(D\times \cU)$ is the following sub-bundle
\begin{equation}
\label{eq:tangentincidence}
  \cT_{(x,u)}\cI \ = \ \left\{(\dot{x},\dot{u})\in \cT_xD\oplus \cT_u\cU \ | \  
  \dot{u}(x) \equiv \dot{x} \ \mathrm{mod} \ T_xZ_u \right\}\,.
\end{equation}
Above, the tangent vector $\dot{u}\in \cT_u \cU$ is viewed as a holomorphic section of the normal bundle $N_{Z_u/D} = \cT D/\cT Z_u$; thus, the the normal field $\dot{u}$ can be evaluated at any $x\in Z_u$ to give a vector in $\cT_x D/\cT_x Z_u$.  Geometrically, $\cT_{(x,u)}\cI$ is the set of tangent vectors $(\dot{x}(0),\dot{u}(0))$ given by curves $(x(t),u(t)) \in D \times \cU$ with $x(t)\in Z_{u(t)}$ and $(x(0),u(0)) = (x,u)$.

\begin{definition}
\label{def:differentialsystems}
Define  sub-bundles $S,E$ and $F$  of $T\cI$ by:
\begin{eqnarray*}
  S_{(x,u)}  & = & 
  \left\{ (\dot{x},\dot{u})\in T_{(x,u)}\cI \ | \ 
          \dot{x}\in T_x Z_u \right\}  \\ 
  & = & \left\{ (\dot{x},\dot{u})\in T_{(x,u)}\cI \ | \ 
                \dot{u}(x) = 0 \right\}  \,,\\
  E_{(x,u)} & = & \left \{(\dot{x},0)\in T_{(x,u)}\cI \right\} 
  \ = \ \left \{(\dot{x},0)\in T_xD\oplus T_u\cU \ | \ 
                \dot{x}\in T_x Z_u \right\} \,,\\
  F_{(x,u)} & = &  \left\{(0,\dot{u})\in T_{(x,u)} \cI \right\} 
  \ = \ \left\{ (0,\dot{u})\in T_x D\oplus T_u \cU \ | \ 
                \dot{u}(x) = 0 \right\} \,.
\end{eqnarray*}
\end{definition}

\noindent Note that the asserted equalities in the definitions follow from the defining equation (\ref{eq:tangentincidence}) of $T\cI$.    
\begin{remark}
The geometric interpretation of these bundles is as follows.  Let   $(x(t),u(t))$ be a  local holomorphic curve in $\cI$.  Then
\begin{blist}
\item $(\dot{x}(t),\dot{u}(t) )\in S$ for all $t$ if and only if  $\dot{x}(t)\in T_{x(t)} Z_{u(t)}$.  In other words, $x$ stays in $Z_u$ to first order.
\item $(\dot{x}(t),\dot{u}(t)) \in E$ for all $t$ if and only if $u(t)$ is constant $\equiv u$ and $x(t)\in Z_u$ for all $t$.  In other words, $E$ is the bundle tangent to the fibers of $p_\cU:\cI\to \cU$.
\item  $(\dot{x}(t),\dot{u}(t)) \in F$ for all $t$ if and only if $x(t)$ is constant  $\equiv x$ and $x\in Z_{u(t)}$ for all $t$.  In other words, $F$ is the bundle tangent to the fibers of the projection $p_D:\cI\to D$.
\end{blist}
\end{remark}

\begin{lemma} \label{lem:S=E+F}
The sub-bundles $E$ and $F$ are integrable, and $S = E\op F$.
\end{lemma}

\begin{lemma}
\label{lem:bracketgeneration}
The sub-bundle $S$ is bracket--generating.
\end{lemma}

Recall that a sub-bundle $S$ of the tangent bundle $T M$ of a manifold $M$  is called \emph{bracket generating} if, for every $x\in M$,  the evaluation map $\cL_x\to T_x M$ is surjective, where $\cL_x$ is the Lie algebra of  germs of vector fields generated by the germs at $x$ of sections of $S$.

Lemmas \ref{lem:S=E+F} and \ref{lem:bracketgeneration} are proved using Lie algebra descriptions of the bundles above.  First, observe that each of the  spaces in \eqref{eq:incidencediagram} is an open set in a corresponding homogenous complex manifold.  In fact,  the diagram \eqref{eq:incidencediagram} embeds in the diagram
\begin{equation}
\label{eq:dualincidencediagram}
 \begin{array}{ccccc} &  & \check{\cI } = G_\bC/(K_\bC\cap P)&  &  \\ &  \swarrow &  & \searrow  &  \\  \check{D} = G_\bC/P &  &  &  & \check{\cU} = G_\bC/K_\bC\end{array}
\end{equation}
Moreover, the bundles $S, E, F$ of Definition (\ref{def:differentialsystems}) extend to homogeneous vector bundles $\check{S}, \check{E}, \check{F}$ over $\check{\cI}$.  Indeed, 
$$
  \cT\check\cI \ = \ 
  G_\bC \times_{K_\bC\cap P} \left(\fg_\bC / \fk_\bC\cap\fp\right)
  \ = \ G_\bC \times_{K_\bC\cap P} \left( \fg_\bC / \fk_0 \op \fk_+ \right) \,.
$$
Use the vector space identification $\fg_\bC/\fk_0 \op \fk_+ \simeq \fk_- \op \fq = \fk_- \op \fq_- \op \fq_+$ to regard $\fk_- \op \fq$ as a $K_\bC\cap P$--module.  Then we mildly abuse notation by writing 
\begin{equation} \label{E:checkTI}
  \cT\check\cI \ = \ 
  G_\bC \times_{K_\bC\cap P} \left( \fk_- \op \fq_- \op \fq_+ \right) \,.
\end{equation}
With respect to this identification, we have 
\begin{eqnarray}
  \nonumber
  \check S & = & G_\bC \times_{K_\bC\cap P} \left( \fk_- \op \fq_+ \right) \,,\\
  \label{E:checkbdls}
  \check E & = & G_\bC \times_{K_\bC\cap P} \fk_- \,,\\
  \nonumber
  \check F & = & G_\bC \times_{K_\bC\cap P} \fq_+ \,.
\end{eqnarray}
The bundles $\cT\cI$, $S$, $E$ and $F$ are the restrictions of $\cT\check\cI$, $\check S$, $\check E$ and $\check F$ to $\cI \subset \check \cI$.

\begin{proof}[Proof of Lemma \ref{lem:S=E+F}]
It is immediate from \eqref{E:checkbdls} that $\check S = \check E \op \check F$; thus $S = E \op F$.  Since $E, F$ are tangent to the fibers of the projections (\ref{eq:incidencediagram}) they are clearly integrable.  Alternatively, to see that $\check E$ and $\check F$ are involutive, it suffices to show that 
$$
  [\fk_- , \fk_-] \,\subset\, \fk_- \quad\hbox{mod}\quad \fk_\bC\cap\fp \,,
  \quad\hbox{and}\quad
  [\fq_+ , \fq_+] \,\subset\, \fq_+ \quad\hbox{mod}\quad \fk_\bC\cap\fp \,.
$$
These two relations are straightforward consequences of the identities in \S\ref{S:LT}.
\end{proof}

The proof of Lemma \ref{lem:bracketgeneration} is more involved.

\section{Proof of Lemma \ref{lem:bracketgeneration}}
\label{sec:proofbracketgeneration}

By \eqref{E:checkTI} and \eqref{E:checkbdls}, the bundle $\check S$ is bracket--generating (equivalently, $S$ is bracket generating) if and only if $\fq_-$ is contained in the algebra generated by $\fk_- \op \fq_+$.  Since, $\fg_- = \fk_- \op \fq_-$, this is equivalent to 
\begin{equation} \label{E:lemB1}
  \hbox{$\fg_-$ is contained in the algebra generated by $\fk_- \op \fq_+$.}
\end{equation}
It will be helpful to define an auxiliary grading element  
$$
  \ttT' \ = \ \sum_{\s_i \in\Delta(\fq_\bC)} \ttT^i \,.
$$
Let $\fg_\bC = \fg'_a \op \fg'_{a-1}\op\cdots\op\fg'_{1-a}\op\fg'_{-a}$ be the $\ttT'$--graded decomposition, cf. \eqref{E:espace}.  By \cite[Theorem 3.2.1(1)]{CS}, 
\begin{center}
  $\fg'_+$ (resp., $\fg_-$) is generated by $\fg'_1$ (resp., $\fg'_{-1}$).
\end{center}
A similar argument implies that 
\begin{center}
  $\fg_+$ (resp., $\fg_-$) is generated by $\fg_1\op\fg_2$ 
  (resp., $\fg_{-1}\op\fg_{-2}$).
\end{center}
Therefore, by $\fg_{-2} \subset \fk_-$ and \eqref{E:lemB1}, $\check S$ is bracket--generating if and only if
\begin{equation} \label{E:lemB2}
  \hbox{$\fg_{-1}$ is contained in the algebra generated by $\fk_- \op \fq_+$.}
\end{equation}
Observe that $\fg_{-1} \subset \fg'_{-1}$, $\fg_0\subset\fg'_0$ (both inclusions usually strict), and 
$$
  \fg'_{-,\teven} \,\subset\, \fg_{-,\teven} \,=\, \fk_- \quad\hbox{and}\quad
  \fg'_{+,\todd} \,=\, \fg_{+,\todd} \,=\, \fq_+ \,.
$$
Therefore, by \eqref{E:lemB2}, to see that $\check S$ is bracket--generating it suffices to show that 
\begin{equation} \label{E:lemB3}
  \hbox{$\fg'_{-1}$ is contained in the algebra generated by 
  $\fg'_{-,\teven} \op \fg'_{+,\todd}$.}
\end{equation}

\begin{proof}[Proof of \eqref{E:lemB3}]
Note that $[\fg'_{1},\fg'_{-2}]$ is a direct sum of (a subset of) root spaces $\fg^{-\a} \subset \fg'_{-1}$.  Let $\Gamma = \{-\a\in\Delta(\fg'_{-1}) \ | \ \fg^{-\a} \not\subset [\fg'_{1},\fg'_{-2}]\}$, and define 
$$
  \ff_{-1} \ = \ \bigoplus_{-\a\in \Gamma} \fg^{-\a} \tand
  \ff_{1} \ = \ \bigoplus_{-\a\in \Gamma} \fg^{\a} \,.
$$
Note that $\fg'_{-1} = [\fg'_1,\fg'_{-2}]$ (and \eqref{E:lemB3} holds) if and only if $\ff_{\pm1}=0$.  Set $\ff_0 = [\ff_1,\ff_{-1}]$.  Define
$$
  \ff \ = \ \ff_1\,\op\,\ff_0\,\op\,\ff_{-1}\,.
$$
We will show that $\ff$ is an ideal of $\fg_\bC$.  Since, by assumption, $\fg_\bC$ is simple, it follows that either $\ff=0$ or $\ff = \fg_\bC$.  If $\ff = \fg_\bC$, then $\fg_\pm = \ff_{\pm1}$ is abelian.  It follows easily that $\ff_{\pm} = \fg'_{\pm} = \fq_\pm$ and $\ff_0 = \fg'_0 = \fk_\bC$.   Therefore, $ G/K$ is Hermitian symmetric, with $(1,0)$-tangent space $\fq_-$.  Moreover, the derivative of the projection $p:G/V\to G/K$ is the projection $\fk_-\oplus \fq_- \to \fq_-$.  Thus (recalling (\ref{E:cpxstr})) $p$ is holomorphic,  contradicting our assumption that $D = G/V$ is non-classical.  Therefore, $\ff=0$, and \eqref{E:lemB3} holds, establishing Lemma \ref{lem:bracketgeneration}.

To prove that $\ff$ is an ideal, we proceed in these steps:

\medskip

\noindent\emph{Step 1}.  Note that $\ff_{-1}$ is the maximal subspace of $\fg'_{-1}$ with the property that $B(\ff_{-1} , [\fg'_{-1},\fg'_2]) = 0$.  Equivalently, 
$$
  0 \ = \ B( \ff_1 , [\fg'_1,\fg'_{-2}] ) \ = \ 
  B( \fg'_1 , [\ff_1,\fg'_{-2}] ) \ = \ B( \fg'_{-2} , [\ff_1 , \fg'_1 ]) \,.
$$
In particular, $\ff_1$ is the largest subspace of $\fg'_1$ with the property that $[\ff_1,\fg'_1] = 0$; equivalently, $[\ff_{-1},\fg'_2] = 0 = [\ff_1,\fg'_{-2}]$.

\medskip

\noindent\emph{Step 2: $\ff$ is a $\fg'_0$--module}.  First observe that $\ff_1$ is a $\fg'_0$--module.  To see this, let $G'_0 = \{ g \in G_\bC \ | \ \tAd_g(\fg'_\ell) = \fg'_\ell \,,\ \forall \ \ell \}$.  Then $G'_0$ is a closed Lie subgroup of $G_\bC$ with Lie algebra $\fg'_0$.  By Step 1, $\ff_1$ is the largest subspace of $\fg'_1$ with the property that $[\ff_1,\fg'_1]=0$.  So,
$$
  0 \ = \ \tAd_{G'_0}[ \ff_1 , \fg'_1] \ = \ 
  [ \tAd_{G'_0} \ff_1 , \tAd_{G'_0} \fg'_1] \ = \ [ \tAd_{G'_0}\ff_1 , \fg'_1] 
$$
implies $\tAd_{G'_0}\ff_1 = \ff_1$.  

An identical argument, with $\ff_{-1}$ in place of $\ff_1$, proves that $\ff_{-1}$ is a $\fg'_0$--module.  It follows from $\ff_0 = [\ff_1,\ff_{-1}]$ and $\tAd_{G'_0}[\ff_1,\ff_{-1}] = [\tAd_{G'_0}\ff_1 \,,\, \tAd_{G'_0}\ff_{-1}]$ that $\ff$ is a $\fg'_0$--module.

\medskip

\noindent\emph{Step 3: $[\fg'_{\pm1},\ff] \subset \ff$}.  To see that $[\fg'_1,\ff]\subset \ff$ we consider each component of $\ff=\ff_1\op\ff_0\op\ff_{-1}$:
\begin{blist}
\item  
By Step 1, $[\fg'_1,\ff_1] = 0$.  
\item 
To see that $[\fg'_1,\ff_{-1}] \subset \ff_0$, let $\ff_0^\perp$ be the Killing orthogonal complement to $\ff_0$ in $\fg'_0$.  Then
\begin{myeqn}{\ast}
  0 \ = \ B(\ff_0,\ff_0^\perp) \ = \ 
  B( [\ff_1 , \ff_{-1}] , \ff^\perp_0 ) \ = \ 
  B(  \ff_1 , [\ff_{-1}, \ff^\perp_0] ) \,. 
\end{myeqn}
By Step 2, $\ff_{-1}$ is a $\fg'_0$--module.  Since $\ff^\perp_0 \subset \fg'_0$, we must have $[\ff_{-1},\ff^\perp_0] \subset \ff_{-1}$.  Then $(\ast)$ yields $[\ff_{-1},\ff^\perp_0] = 0$.  It follows that 
$$
  B( [\fg'_1 , \ff_{-1}] \,,\, \ff^\perp_0 ) \ = \ 
  B( \fg'_1 \,,\, [\ff_{-1} , \ff^\perp_0] ) \ = \ 0 \,,
$$
yielding $[\fg'_1,\ff_{-1}] \subset \ff_0$.  
\item 
The Jacobi identity implies that $[\fg'_1 , \ff_0] = [\fg'_1,[\ff_1,\ff_{-1}]] \subset \ff$.  
\end{blist}
We conclude that $[\fg'_1 , \ff] \subset \ff$.  A similar argument yields $[\fg'_{-1},\ff] \subset \ff$, completing Step 3.

\medskip

\noindent\emph{Step 4: Induction}.    Suppose that $[\fg'_\ell,\ff] \subset \ff$ for some $\ell > 0$.  Since $\fg'_1$ generates $\fg'_+$, we have $\fg'_{\ell+1} = [\fg'_\ell,\fg'_1]$. Then the Jacobi identity yields $[\fg'_{\ell+1},\ff] \subset \ff$.  Thus $[\fg'_+ , \ff] \subset \ff$.  A similar argument yields $[\fg'_-,\ff] \subset \ff$.

\medskip

\noindent\emph{Fini}.  This completes the proof that $\ff$ is an ideal of $\fg_\bC$, and establishes \eqref{E:lemB3}, and thus Lemma \ref{lem:bracketgeneration}.
\end{proof}

\section{Proof of Theorem \ref{thm:Zconnectedness}}

Let $E, F, S$ be as in Definition \ref{def:differentialsystems}, and let $\cE,\cF,\cS$ be their respective  spaces of smooth sections.  We are interested in the collection $\cE\cup \cF$ of vector fields on $\cI$ and in the pseudogroup $\cG$ generated by their flows:

\begin{definition}
\label{def:pseudogroup} 
Let $\cG$ be the collection of all diffeomorphisms of $\cI$ of the form
\begin{displaymath}
\Phi(\vec{X},\vec{t}) = \exp(t_k X_k)\circ \dots \circ \exp(t_2 X_2 )\circ \exp(t_1 X_1)
\end{displaymath}
where $\vec{X} = (X_1,\dots X_k)$, each $X_i \in \cE\cup\cF$; $\vec{t} = (t_1,\dots t_k)$, each $t_i\in \bR$, and $k$ is a natural number.

Given $(x,u)\in\cI$, by the \emph{orbit of $(x,u)$ under $\cG$}, we mean the set $\cG (x,u) = \{\Phi(\vec{X},\vec{t})(x,u): \mbox{all }\vec{X}, \vec{t}\mbox{ as above}\}$.
\end{definition}

We have the following proposition:

\begin{proposition}
\label{prop:orbit}
For any $(x_0,u_0)\in\cI$, the orbit $\cG (x_0,u_0) = \cI$.
\end{proposition}

\begin{proof}
Let $\cL(\cE\cup\cF)$ denote the Lie subalgebra of the algebra $C^\infty (T\cI)$ of smooth  vector fields on $\cI$ generated by $\cE\cup\cF$.  For each $(x,u)\in \cI$, let $\cL_{(x,u)}(\cE\cup\cF) \subset T_{(x,u)}\cI$ denote the set of evaluations at $(x,u)$ of all the vector fields in $\cL(\cE\cup\cF)$.  A standard theorem in Control Theory, going back to  \cite{Stefan, Sussmann}, asserts: 
\begin{quote}
  If $\cL(\cE\cup\cF) = \cT_{(x,u)}\cI$, for all $(x,u)\in \cI$, 
  then $\cG$ has only one orbit; namely $\cI$.   
\end{quote}
(Cf. Theorem 3 of \cite[Section 2.3]{Jurdjevic}.)  So to prove the proposition, it suffices to show
\begin{equation} \label{E:prop}
\hbox{For all $(x,u)\in\cI$, we have $\cL_{(x,u)}(\cE\cap\cF) = T_{(x,u)}\cI$.}
\end{equation}

By Lemma \ref{lem:S=E+F}, $E\oplus F = S$.  So, the Lie algebra generated by $\cE\cup\cF$ is the same as the Lie algebra $\cL(\cS)$ generated by $\cS$.  By Lemma \ref{lem:bracketgeneration}, $S$ is bracket generating.  Thus, for all $(x,u)\in \cI$, $\cL_{(x,u)}(\cS) = T_{(x,u)}\cI$; this establishes \eqref{E:prop}. 
\end{proof}

Now we can complete the proof of Theorem \ref{thm:Zconnectedness}.  Let $(x_0,u_0)\in\cI$,  let $\Phi(\vec{X},\vec{t})\in\cG$ be as in Definition \ref{def:pseudogroup}, and, for each $j = 1,\dots, k$, let
\begin{equation}
\label{eq:staircase}
(x_j,u_j) = \exp(t_j X_j)\circ\dots\circ\exp(t_1 X_1)(x_0,u_0)
\end{equation}
so that 
\begin{equation}
\label{eq:onestep}
(x_{j+1},u_{j+1}) = \exp(t_{j+1} X_{j+1})(x_j,u_j).
\end{equation}
Note that each step (\ref{eq:onestep}) is obtained by flowing along the vector field $X_{j+1}$:
let $\gamma(t)$ be the curve $\exp(tX_{j+1})(x_j,u_j)$.  Then $\gamma(0) = (x_j,u_j)$ and $\gamma(t_{j+1}) = (x_{j+1},u_{j+1})$.   We will describe each step (\ref{eq:onestep}) by saying \emph{$(x_j,u_j)$ flows along $X_{j+1}$ to $(x_{j+1},u_{j+1})$}.

Since $X_{j+1}\in\cE\cup\cF$, in each step (\ref{eq:onestep}) either $u$ is constant or $x$ is constant.  That is, either:

\begin{blist}
\item $u_j = u_{j+1}$ and $x$ flows from $x_j$ to $x_{j+1}$ inside $Z_{u_j}$, if $X_{j+1}\in\cE$.
\item $x_j = x_{j+1}$ and $Z_u$ flows from $Z_{u_j}$ to $Z_{u_{j+1}}$ keeping $x_j$ fixed, if $X_{j+1}\in\cF$.
\end{blist}
Therefore the points $x_0$ and $x_k$ are joined by a chain of subvarieties $Z_{u_0}\dots Z_{u_k}$ where each subvariety meets the next.   (Maybe one is identical to the next, in which case we may want to eliminate repetitions).  Since $x_0, x_k\in D$ are arbitrary, Theorem \ref{thm:Zconnectedness}  is proved.

\section{Proof of Theorem \ref{thm:nonalgebraic}}

Let $D$ be a non-classical domain, let $\Gamma\subset G$ be an infinite, finitely generated discrete subgroup.    Then $X = \Gamma\backslash D$ is a normal analytic space, and we want to prove that $X$ does not have an algebraic structure.

\subsection{Reduction to torsion-free $\Gamma$}
\label{subsec:reduction}

Since $\Gamma$ is a finitely generated linear group, it has a torsion-free subgroup $\Gamma'$  of finite index 
(Proposition 2.3 of \cite{Borel}).    Then $\Gamma'$ acts freely on $D = G/V$ and the quotient space $X' = \Gamma'\backslash D$ is a complex manifold with fundamental group $\Gamma'$.  The natural projection $\pi:X'\to X$ is a finite analytic map, that is, a proper analytic map with finite fibers.

\begin{lemma}
\label{lem:codimension}
Let $\gamma\in G$ be an element of finite order, $\gamma\ne e$.  Then the fixed point set of $\gamma$ is a complex submanifold of $D$ of  complex codimension at least two.
\end{lemma}
\begin{proof}

If $\gamma$ has a fixed point in $G/V$ then it is conjugate in $G$ to an element of $V$, so we may assume $\gamma\in V$.  If a group $H$ acts on a space $A$ and $h\in H$, let  $F(h, A)$ denote the fixed point set of $h$ in $A$. Since $\gamma$ is a holomorphic map of $D$ and is an isometry of a Hermitian metric (\S 9 of \cite{GS}), $F(\gamma,D)$ is a complex submanifold of $D$.  We claim  that  $F(\gamma,D)$ is fibered as follows by the restriction of the projection $p:D = G/V\to G/K$: 
\begin{center} \setlength{\unitlength}{5pt}
\begin{picture}(25,12)
\put(0,10){$F(\gamma, K/V)$} \put(12,10.5){\vector(1,0){5}}
\put(18,10){$F(\gamma,G/V)$} \put(23,9){\vector(0,-1){5.5}} \put(24,6.5){$p$}
\put(18,0){$F(\gamma,G/K)$.}
\end{picture}
\end{center}
This is easily checked by writing explicitly $p|F(\gamma,D)$:
$$
 \{gV\in G/V: g^{-1}\gamma g \in V\} \to  \{gK\in G/K: g^{-1}\gamma g \in K\}
$$
which takes $gV$ to $gK$.  This map is surjective because the fiber over $gK\in  F(\gamma,G/K)$ is 
$$
\{gkV:k\in K\hbox{ and } k^{-1}g^{-1}\gamma g k\in V\} = gF(g^{-1}\gamma g,K/V),
$$ 
which is not empty, because $K$ is compact and $T\subset V$, so $g^{-1} \gamma g$ is conjugate to an element of $T$.  Thus $p|_{F(\gamma,G/V)}$ is a fibration of  $F(\gamma,D)$  over $F(\gamma,G/K)$ with fibers biholomorphic to $F(\gamma,K/V)$.  

Recall we are assuming that $G$ is a simple  adjoint group.  Since it has no non-trivial normal subgroups, it must act effectively on $D = G/V$ and on the symmetric space $G/K$. But $K$ must have non-trivial center  since it contains the Cartan involution. Recall that $K/V$ is always a positive-dimensional complex manifold, since  $D$ is non-classical.  In particular, $K$ is non-abelian.

 There are two cases to consider:

\noindent\emph{Case 1: $\gamma$  does not fix $K/V$ pointwise.}   Then the real codimension of $F(\gamma, K/V)$ in $K/V$ is at least two, since it is a proper complex submanifold of the positive-dimensional complex manifold $K/V$.  The real codimension of $F(\gamma,G/K)$ is at least one, since it is a proper submanifold of $G/K$. Therefore the real codimension of $F(\gamma,D)$ is at least three, so its complex codimension is at least two.

\noindent\emph{Case 2: $\gamma$ fixes $K/V$ pointwise.}  Then 
$\gamma\in N = \cap_{k\in K} k V k^{-1}$, which is a non-trivial proper normal subgroup of $K$ containing the center $Z(K) =\cap_{k\in K} kTk^{-1}$ of $K$, that is  $Z(K)\subset N\subsetneq V \subsetneq K$.  

Recall that $K$ is connected and is either a simple group, or finitely covered by a product of compact simple groups $K_1,K_2,\dots$ or, exactly when $G/K$ is Hermitian symmetric, it is finitely covered by a product $S^1\times K'$ where $K'$ is a product of simple groups and the image of $S^1$ is the center of $K$.

We claim that $Z_K(\gamma)$, the centralizer of $\gamma$ in $K$, always contains a  non-abelian subgroup $L$ of $K$ which is a product of simple factos of $K$.    To prove this, observe that either $N\subset Z(K)$, in which case $Z_K(\gamma) = K$, or $N^0$ (the identity component of $N$)  consists of a proper product of factors,  omitting at least one simple factor of $K$.  (In the Hermitian symmetric case we use the fact that $N$ contains $Z(K)$ to see that the $S^1$-factor must be in $N^0$).  Let $L$ be the product of all the omitted factors.  Then $L$ is non-abelian and  from $K = LN^0$, $L\cap N^0$ central, and $N=\cup l_i N^0$ (union of finitely many cosets $l_iN^0$, where $l_i\in L$), it is easy to see that    $N = C N^0$ for some finite central subgroup $C$.    Then $N$, hence $\gamma$, commutes with $L$, that is, $L\subset Z_K(\gamma)$.

Since $Z_K(\gamma)$ leaves $F(\gamma,G/K)$ invariant, so does $L$.
Then the tangent space $T_{eK} F(\gamma,G/K)$ is a proper $L$-invariant subspace of $T_{eK}G/K$.  It has a non-zero $L$-invariant complement.   Since $G/K$ is irreducible, the action of $K$ on the tangent space $T_{eK}G/K$  is irreducible.  This implies that the factor $L$ of $K$ cannot act trivially on any non-zero subspace of $T_{eK}G/K$.  (If it did, the subspace of $L$-invariant vectors in  $T_{eK}G/K$ would be a non-zero, proper subspace stable under $K$, contradicting irreducibility.)  

Therefore the $L$-invariant complement of $T_{eK}F(\gamma,G/K)$ is a non-trivial representation of $L$.  But a non-trivial representation of a compact, connected  simple group on a real vector space must have dimension at least three, so the real codimension of $F(\gamma,G/K)$ in $G/K$ is at least three.  Therefore the real codimension of $F(\gamma,G/V)$ must be at least three, so its complex codimension at least two.
\end{proof}

Applying this lemma to the elements of finite order in $\Gamma$, we see  that there are analytic subsets $Y, Y'$ of codimension at least two,  $Y'\subset X'$  and  $Y\subset X$, so that  $\pi:X'\setminus Y'\to X\setminus Y$ is unramified.  Moreover $Y$ is the singular set of the normal analytic space  $X$.

  Suppose $X$ has an algebraic structure, meaning that  there is an algebraic variety $W$ so that $X = W^\tanalytic$ is the analytic space associated to $W$.  Then $Y = \Sigma^\tanalytic$, where $\Sigma\subset W$ is the singular set of $W$.  Let $U = W\setminus \Sigma$ be the set of regular points of $W$.  Since $\pi|_{X'\setminus Y'}$ is finite \'etale, by
   the \lq\lq Riemann Existence Theorem"  (Theorem XII 5.1 of \cite{SGA}),  $X'\setminus Y'$ would also have an algebraic structure. In other words, there would be an algebraic variety $U'$ so that $U'^\tanalytic = X'\setminus Y'$ and a finite \'etale map $\rho:U'\to U$ with $\rho^\tanalytic = \pi$.
   
  Let $j:U\to W$ be the inclusion. \emph{Since  $\Sigma$ has codimension at least two},  $j_*\rho_*\cO_{U'}$   is a coherent sheaf of algebras over $\cO_{W}$ whose associated analytic sheaf on $X = W^\tanalytic$  agrees with $\pi_*\cO_{X'}$ on $X\setminus Y$, hence  on all of $X$.  Then $W' = \mathrm{ Spec}_W (j_*\rho_*\cO_{U'})$ is a scheme of finite type over $\bC$ with $W'^\tanalytic = X'$.  
  
   In other words, if $X$ were algebraic,  $X'$ would also be algebraic.  So we may  assume that our discrete group $\Gamma$ is torsion--free.

\subsection{The Shafarevich Map}

Assume from now on, that  $X = \Gamma \backslash D$  where $\Gamma$ has no torsion, thus $X$ is a complex manifold with fundamental group $\Gamma$.  Suppose that  $X$ were an algebraic variety.    We will derive a contradiction by using Koll\'ar's Shafarevich map.  We begin by recalling the facts that we need from \cite{Kollar}.

First, according to Definition 3.5  of \cite{Kollar},  (specialized to the case $H = \{id\}$), if $X$ is a normal algebraic variety, a \emph{Shafarevich variety} $Sh(X)$ and \emph{Shafarevich map} $sh_X$ are a normal variety $Sh(X)$ and a rational map $sh_X: X\dashrightarrow Sh(X)$ with the properties: that $sh_X$ has connected fibers, and that there is a countable collection  of closed subvarieties of $D_i \subset X$,  $D_i\ne X$, so that \emph{for any closed, irreducible subvariety $Z\subset X$  not contained in $\cup D_i$} the following holds:
\begin{equation}
\label{eq:defineshafarevich}
sh_X(Z) = \mbox{ point \  if and only if } \tim\{\pi_1(\bar Z)\to \pi_1(X)\}\ \mbox{ is finite,}
\end{equation}
where $\bar Z$ denotes the normalization of $Z$. 

\begin{definition}[Definition 2.1 of \cite{Kollar}]
\label{def:normalcycle}
A \emph{normal cycle} in $X$  means an irreducible normal variety $W$ together with a finite morphism $w:W\to X$ that is birational to its image.
\end{definition}
 
\begin{theorem}[Theorem 3.6 of \cite{Kollar}]
\label{thm:kollar}
Let $X$  be a normal algebraic variety.  Then a Shafarevich variety and map exist, unique up to birational equivalence.  Moreover, for every choice of $Sh(X)$ within its birational equivalence class, there are Zariski open subsets $X_0\subset X$ and $Y^0\subset Sh(X)$ so that
\begin{a_list}
\item  
$sh_X:X^0\to Y^0$   is everywhere defined,
\item 
Every fiber of $ sh_X | X^0$ is closed in  $X$,
\item
$sh_X|X^0$  is a topologically locally trivial fibration.
\item 
Let $y\in Y^0$ be very general, let $X^0_y$ be the fiber of $sh_X|X^0$ over $y$, and let $w:W\to X$ be a normal cycle with $\tim\{\pi_1(W)\to \pi_1(X)\}$ finite.  If $\tim\{w\}\cap X^0_y\ne\emptyset$, then $\tim\{w\}\subset X^0_y$.
\end{a_list}
\end{theorem}

\begin{remark}
\label{rm:fibers}
The fourth statement does not appear explicitly in the statement of Theorem 3.6 of \cite{Kollar}, but  it appears in  Corollary 3.4 and in the proof of Theorem 3.6, and is used in the proof of Theorem 4.13.   It is a more detailed version of (\ref{eq:defineshafarevich}).  (We have also omitted statements regarding proper varieties, since our interest is in the non-proper situation.)
\end{remark}

\begin{proof}[Proof of Theorem \ref{thm:nonalgebraic}]
We now suppose that the complex manifold $X = \Gamma\backslash D$ has a compatible algebraic  structure and derive a contradiction by using Theorem \ref{thm:kollar}.   We first prove that $sh_X|X^0$ is constant.   

Let $\pi:D\to X$ be the projection.  For any  of our compact subvarieties  $Z_u\subset D$, the map $\pi|Z_u:Z_u\to X$ is an immersion, therefore it is finite and birational to its image, in other words, it is a normal cycle in the sense of Definition \ref{def:normalcycle}.   

Let $y_1,y_2\in Y^0$ be very general points.  Choose $x_1,x_2\in D$ so that $\pi(x_1)\in X^0_{y_1}$ and $\pi(x_2)\in X^0_{y_2}$.  By Theorem \ref{thm:Zconnectedness} there is a chain $u_1,\dots,u_k\in \cU$ so that, letting $Z_i $ denote $Z_{u_i}$,  $x_1\in Z_1, x_2\in Z_k$ and $Z_i\cap Z_{i+1} \ne \emptyset$.  Choose a point $z_i\in Z_i\cap Z_{i+1}$.  Let  $w_i$ denote the normal cycle  $w_i = \pi|Z_i:Z_i\to X$.  Recall that $Z_i = K/V$ is simply connected.  By (d) of Theorem \ref{thm:kollar},  $\tim\{w_1\}\subset X^0_{y_1}$, hence $w_1(z_1)\in X^0_{y_1}$, hence $\tim\{w_2\}\subset X^0_{y_1}$.  Continuing this way, we get $w_k(x_2)\in X^0_{y_1}$.  But by definition $w_k(x_2)\in X^0_{y_2}$.  Hence $y_1 = y_2$ for any two very general points in $Y^0$.  Thus $sh_X|X^0:X^0\to Y^0$ is the constant map.

Now, by Theorem \ref{thm:kollar}(b), the single fiber $X^0$ of $sh_X$ is closed in $X$, hence $X^0 = X$, the rational map $sh_X$ is everywhere defined on $X$ and $sh_X(X)$ is a point.  Therefore
 (\ref{eq:defineshafarevich}),  gives that $\pi_1(X)$ is finite.  But $\pi_1(X) \cong \Gamma$ is infinite, thus contradicting the algebraicity of $X$. 
\end{proof}

\end{document}